\documentclass[12pt]{amsart}

\usepackage{amssymb,amsmath,latexsym,comment,color,cite}

\usepackage[left=3cm, right=3cm, top=3cm, bottom=3cm]{geometry}
\newtheorem{theorem}{Theorem}[section]

\newtheorem{corollary}[theorem]{Corollary}

\newtheorem{lemma}[theorem]{Lemma}

\newtheorem{problem}[theorem]{Open problem}
\newtheorem{proposition}[theorem]{Proposition}

\theoremstyle{definition}
\newtheorem{remark}[theorem]{Remark}
\newtheorem{definition}[theorem]{Definition}
\newtheorem{example}[theorem]{Example}

\begin{document}

\title[Arithmetic-geometric mean, additive, and multiplicative contractions]{Arithmetic-geometric mean, additive, and multiplicative contractions: new generalizations of the Banach Contraction Principle}

\author[I. S. Singh]{Irom Shashikanta Singh}

\address{Department of Mathematics, Manipur University,Imphal West, 795003,
Manipur, India}

\email{shashikanta.phd.math@manipuruniv.ac.in}

\author[Y. M. Singh]{Yumnam Mahendra Singh}

\address{Department of Basic Sciences and Humanities,
Manipur Institute of Technology,
Canchipur, Imphal West, 795003,
Manipur, India}

\email{ymahenmit@rediffmail.com}

\author[E. Petrov]{Evgeniy Petrov}

\address{Institute of Applied Mathematics and Mechanics of the NAS of Ukraine,
Batiuka str. 19, 84116 Slovyansk, Ukraine}

\email{eugeniy.petrov@gmail.com}

\author[R. Salimov]{Ruslan Salimov}

\address{Institute of Mathematics of the NAS of Ukraine,
Tereschenkivska str. 3,  01024 Kiev, Ukraine}

\email{ruslan.salimov1@gmail.com}



\date{}

\dedicatory{}

\commby{}


\begin{abstract}
We introduce new contraction conditions based on classical inequality between arithmetic and geometric means. By incorporating an auxiliary semimetric $\delta$, we define arithmetic-geometric mean, multiplicative-type, and additive-type contractions. Connections between these types of contractions are found. Fixed point theorems are proved in the case of continuity of the above mentioned contractions. Under suitable regularity conditions on $\delta$ (such as being d-regular, strongly d-regular, or d-lower bounded) we obtain constructive corollaries. Various examples demonstrating our results are constructed. It is shown that with certain caveats fixed point theorem for additive-type mappings is equivalent to the fixed point theorem for perturbed metric spaces, which were recently introduced by M. Jleli and B. Samet.
\end{abstract}

\maketitle

\section{Introduction and preliminaries}
The Banach fixed-point theorem represents a pivotal contribution to mathematics, particularly within functional analysis. Originating from the work of Stefan Banach, a prominent Polish mathematician in the early 20th century, this theorem remains a foundational tool in mathematical discourse. The Banach fixed-point theorem~\cite{banach} states the following. 
\begin{theorem}\label{ThBanach}	
    Let $( X ,d)$ be a complete metric space and $T  \colon   X  \rightarrow  X  $ be a self-mapping satisfying
    \begin{equation*}
        d(T  x,T  y)\le \alpha  d(x,y),
    \end{equation*}
    for all $x,y \in  X  $, where $\alpha \in [0,1)$ is a constant. Then, $T $ is a Picard operator.
\end{theorem}  
Let $( X ,d)$ be a metric space and $T \colon  X  \rightarrow  X  $ be a self-mapping. Recall that $T $ is called \emph{Picard operator}~\cite{picard operator} if $T $ has a unique fixed point $u \in  X  $ and for any $x \in  X  $, the sequence of iterates $(T ^nx)$ converges to $u$.

The Banach contraction principle has undergone extensive generalizations across various mathematical dimensions over the past decades. It is possible to distinguish at least three fundamental directions in the development of this theory. In the first case, the contractive nature of the mapping is weakened by relaxing the functional requirements, as demonstrated in the works of, e.g.,~\cite{BW69,Ki03,Mk69,Pr20,Ra62,Su06,Wa12,P23}. Another significant approach involves weakening the underlying topology of the space, for instance, by considering semimetric, $b$-metric, or quasi-metric structures~\cite{Ba00, JS15,Fr00,HZ07,KKR90,Ta74,SIIR20}. Finally, a third category comprises multi-valued generalizations that extend the principle to set-valued contractions, see, e.g.,~\cite{Na69,Ma68,AK72}. Such extensions are remarkably numerous and, as a rule, are aimed at establishing the criteria for the existence and uniqueness of a fixed point.

Let $X$ be a nonempty set. Recall that a mapping  $d\colon X\times X\to \mathbb{R}^+$, $\mathbb{R}^+=[0,\infty)$ is a \emph{metric} if for all $x,y,z \in X$ the following axioms hold:
\begin{itemize}
  \item [(i)] $(d(x,y)=0)\Leftrightarrow (x=y)$,
  \item [(ii)] $d(x,y)=d(y,x)$,
  \item [(iii)] $d(x,y)\leqslant d(x,z)+d(z,y)$.
\end{itemize}
The pair $(X,d)$ is called a \emph{metric space}. If only axioms (i) and (ii) hold then $d$ is called a \emph{semimetric}. A pair $(X,d)$, where  $d$  is a semimetric on $X$, is called a \emph{semimetric space}.
Such spaces were first examined by Fr\'{e}chet in~\cite{Fr06}, where he called them ``classes (E)''. Later these spaces attracted the attention of many mathematicians ~\cite{Ch17,Ni27,Wi31,Fr37,PS22}.

For non-negative real numbers $a$ and $b$, the arithmetic mean (A.M.) and geometric mean (G.M.) are defined as $\frac{a+b}{2}$ and $(ab)^{1/2}$, respectively. These means satisfy the classical inequality
\begin{equation}\label{am-gm}
    (ab)^{1/2} \le \frac{a+b}{2}.
\end{equation}
Motivated by the relationship between A.M. and G.M. given by (\ref{am-gm}), we introduce the concept of arithmetic-geometric mean contraction (A.M.-G.M. contraction), which generalizes the classical Banach contraction principle.

\begin{definition}
      Let \( ( X ,d) \) be a metric space. A self-mapping  \( T   \colon   X   \to  X   \) is said to be \emph{A.M.-G.M. contraction} if there exists $\alpha  \in [0,1)$ and a semimetric $\delta\colon  X\times X \to \mathbb{R}^+$ such that
\begin{equation}\label{am-gm-contr}
\frac{d(T  x, T  y)+\delta(T  x,T  y)}{2} \leq \alpha  \big(d(x, y)\delta(x,y)\big)^\frac{1}{2},
\end{equation}
 for all   $x, y \in  X  $.
\end{definition}

\begin{remark}
    Note that if we set $\delta=d$, then A.M.-G.M. contraction reduces to Banach contraction.
\end{remark}

Applying inequality~(\ref{am-gm}) to the left  hand side of~(\ref{am-gm-contr})  we obtain the following
$$
\big(d(Tx, Ty)\delta(Tx,Ty)\big)^\frac{1}{2} \leq \alpha  \big(d(x, y)\delta(x,y)\big)^\frac{1}{2}.
$$
Taking the square of both sides gives
$$
d(Tx, Ty)\delta(Tx,Ty) \leq \alpha^2  d(x, y)\delta(x,y).
$$
This leads to the following definition.

\begin{definition}
      Let \( ( X ,d) \) be a metric space. A self-mapping  \( T   \colon   X   \to  X   \) is said to be a \emph{multiplicative-type contraction} if there exists $\alpha  \in [0,1)$ and a semimetric $\delta\colon  X\times X \to \mathbb{R}^+$ such that
\begin{equation}\label{multi-contr}
d(Tx, Ty)\delta(Tx,Ty) \leq \alpha  d(x, y)\delta(x,y),
\end{equation}
 for all   $x, y \in  X  $.
\end{definition}

\begin{remark}
It is clear that every A.M.-G.M. contraction with the coefficient $\alpha$ and semimetric $\delta$ is a multiplicative-type contraction with the coefficient $\alpha^2$ and the same semimetric $\delta$. Hence, the class of A.M.-G.M. contractions is a subclass of  multiplicative-type contractions, i.e.,
  $$
  \{\text{A.M.-G.M. contractions}\} \subseteq  \{\text{multiplicative-type contractions}\}.
  $$
\end{remark}

Applying inequality~(\ref{am-gm}) to the right hand side of~(\ref{am-gm-contr}) and multiplying by 2 leads to the following.

\begin{definition}\label{d16}
      Let \( ( X ,d) \) be a metric space. A self-mapping  \( T   \colon   X   \to  X   \) is said to be \emph{additive-type contraction} if there exists $\alpha  \in [0,1)$ and a semimetric $\delta\colon  X\times X \to \mathbb{R}^+$ such that
\begin{equation}\label{additive-contr}
{d(T  x, T  y)+\delta(T  x,T  y)}\leq \alpha  ({d(x, y)+\delta(x,y)}),
\end{equation}
 for all   $x, y \in  X  $.
\end{definition}

\begin{remark}
  It is clear that the class of A.M.-G.M. contractions is a subclass of additive-type contractions, i.e.,
  $$
  \{\text{A.M.-G.M. contractions}\} \subseteq  \{\text{additive-type contractions}\}.
  $$
\end{remark}

\section{Arithmetic-geometric mean contractions}\label{sec2}

In this section, we study the properties of A.M.–G.M. contractions and establish a fixed-point theorem for these mappings.

\begin{proposition}\label{p23}
    Let \( ( X ,d) \) be a metric space. Then, every A.M.-G.M. contraction with the coefficient $\alpha>0$ is Banach contraction if for all $x, y \in  X $ the inequality $\delta(x, y)\le \beta d(x, y)$ holds with  $0 \le \beta < \frac{1}{4\alpha^2}$.
\end{proposition}
\begin{proof}
Let $T \colon   X \to  X $ be an A.M.-G.M. contraction with the contractions coefficient $\alpha$.

Clearly, the following inequality holds

    \begin{equation}\label{5}
        \frac{d(T  x, T  y)}{2} \le \frac{d(T  x, T  y) +\delta(T  x, T  y)}{2}.
    \end{equation}
It follows from~(\ref{am-gm-contr}) and \eqref{5} that
    \begin{equation}\label{6}
        \frac{d(T  x, T  y)}{2} \le \alpha \big(d(x, y)\delta(x,y)\big)^\frac{1}{2}.
    \end{equation}

    Since $\delta(x, y)\le \beta d(x, y)$, \eqref{6} becomes
    \begin{equation*}
        d(T  x, T  y) \le 2\alpha \beta^\frac{1}{2}d(x, y),
    \end{equation*}
    which completes the proof, since $0\leqslant 2\alpha \beta^\frac{1}{2} <1$.
\end{proof}
    
\begin{theorem}\label{am-gm-contr th}
    Let $( X ,d)$ be a complete metric space and $T  \colon   X  \rightarrow  X  $ be a continuous A.M.-G.M. contraction. Then $T$ is a Picard operator.
\end{theorem}
\begin{proof}
    Suppose that $x_0 \in  X  $. We now define a sequence $(x_n)$ for all $n > 0$ by using $x_{n} = T^{n}x_0$. 
    	
Taking $x = x_0$ and $y = x_1$ in the inequality (\ref{am-gm-contr}), we have
\begin{equation*}
\begin{split}
    \frac{d(x_1,x_2)+\delta(x_1,x_2)}{2} \le \alpha  \big(d(x_0,x_1)\delta(x_0,x_1)\big)^\frac{1}{2}\le \alpha  \frac{d(x_0,x_1)+\delta(x_0,x_1)}{2} ,
    \end{split}
\end{equation*}
        which is equivalent to 
        \begin{equation}\label{2.2}
           d(x_1,x_2)+\delta(x_1,x_2) \le \alpha  \big(d(x_0,x_1)+\delta(x_0,x_1)\big).
        \end{equation}
        
         Taking $x = x_1$ and $y = x_2$ in the inequality (\ref{am-gm-contr}), we have
\begin{equation*}
\begin{split}
    \frac{d(x_2,x_3)+\delta(x_2,x_3)}{2} \le \alpha  \big(d(x_1,x_2)\delta(x_1,x_2)\big)^\frac{1}{2}\le  \alpha \frac{d(x_1,x_2)+\delta(x_1,x_2)}{2},
    \end{split}
\end{equation*}
        which is equivalent to 
        \begin{equation}\label{2.3}
           d(x_2,x_3)+\delta(x_2,x_3) \le \alpha  \big(d(x_1,x_2)+\delta(x_1,x_2)\big).
        \end{equation}

        From (\ref{2.2}) and (\ref{2.3}), we have
\begin{equation*}
           d(x_2,x_3)+\delta(x_2,x_3) \le \alpha ^2 \big(d(x_0,x_1)+\delta(x_0,x_1)\big).
        \end{equation*}

        Continuing in this way, we obtain 
\begin{equation}\label{2.4}
           d(x_n,x_{n+1})+\delta(x_n,x_{n+1}) \le \alpha ^n \big(d(x_0,x_1)+\delta(x_0,x_1)\big).
        \end{equation}

Using triangle inequality and (\ref{2.4}), we obtain that for all $n\ge 0$ and $m>n$,
\begin{equation*}
    \begin{split}
        d(x_n,x_m) \le& d(x_{n},x_{n+1})+d(x_{n+1},x_{n+2})+\dots +d(x_{m-1},x_{m})\\
        \le& d(x_{n},x_{n+1})+\delta(x_{n},x_{n+1})+d(x_{n+1},x_{n+2})+\delta(x_{n+1},x_{n+2})\\
        &+\dots+d(x_{m-1},x_{m})+\delta(x_{m-1},x_{m})\\
        \le& \alpha ^n \big(d(x_0,x_1)+\delta(x_0,x_1)\big)+\alpha ^{n+1}\big(d(x_0,x_1)+\delta(x_0,x_1)\big)\\
        &+\dots +\alpha ^{m-1}\big(d(x_0,x_1)+\delta(x_0,x_1)\big)\\
        =&\alpha ^n\left(1+\alpha +\alpha ^2+\dots \alpha ^{m-n-1}\right)\big(d(x_0,x_1)+\delta(x_0,x_1)\big)\\
        =& \alpha ^n\frac{1-\alpha ^{m-n}}{1-\alpha }\big(d(x_0,x_1)+\delta(x_0,x_1)\big).
    \end{split}
\end{equation*}

Since $0\le \alpha  <1$, we have $1-\alpha ^{m-n}<1$ and consequently 
\begin{equation}\label{e11}
    d(x_n,x_m) \le \alpha ^n\frac{1}{1-\alpha }\big(d(x_0,x_1)+\delta(x_0,x_1)\big).
\end{equation}

To demonstrate that the sequence $(x_n)$ is Cauchy based on this bound, select $\varepsilon>0$ and then select $N\ge 1$ so that $\frac{\alpha ^N}{1-\alpha }\big(d(x_0,x_1)+\delta(x_0,x_1)\big)< \varepsilon$. Then, for any $m>n\ge N$
       \begin{equation*}
           d(x_n,x_m)\le \frac{\alpha ^n}{1-\alpha }\big(d(x_0,x_1)+\delta(x_0,x_1)\big) \le \frac{\alpha ^N}{1-\alpha }\big(d(x_0,x_1)+\delta(x_0,x_1)\big)<\varepsilon.
       \end{equation*}

       This proves that $\{x_n\}$ is Cauchy. Since $( X ,d)$ is complete metric space, there exist $x^*\in  X  $ such that $x_n \rightarrow x^*$.

Since $T$ is continuous we have
$$
T  x^*=\lim_{n \to \infty} T  x_n= \lim_{n \to \infty} x_{n+1}=x^*.
$$
So, $x^*$ is a fixed point of $T $.

        Suppose there exist two different fixed points  $u$ and $v$ of $T $ such that $T  u=u\ne v=T  v$. Then, by~(\ref{am-gm-contr}) and by~(\ref{am-gm}) we have
        \begin{equation*}
            \begin{split}
                \frac{d(T  u,T  v)+\delta(T  u,T  v)}{2}\le& \alpha \big(d(u,v)\delta(u,v)\big)^\frac{1}{2}\\
                \le& \alpha  \frac{d(u,v)+\delta(u,v)}{2},
            \end{split}
        \end{equation*}
        which is equivalent to
        $$ d(u,v)+\delta(u,v)\le \alpha  \big(d(u,v)+\delta(u,v)\big).$$

        This implies that $d(u,v)+\delta(u,v)=0$ as $0\le \alpha <1$, which is a contradiction, since $d(u,v)\neq 0$. Hence, $T $ has a unique fixed point in $ X $.
\end{proof}
\begin{remark}
Letting $m\to \infty$ in~(\ref{e11}) we obtain the estimation of the distance from the fixed point $x^*$ to $n$-th iterate of the operator $T$:
$$
    d(x_n,x^*) \le \frac{\alpha ^n}{1-\alpha }\big(d(x_0,x_1)+\delta(x_0,x_1)\big).
$$
\end{remark}

In the following example, we show that in general case the class of mappings described in Theorem ~\ref{am-gm-contr th} is broader than the class of Banach contractions. 
\begin{example}\label{ex1}
 Let \(  X   = \{x_1, x_2, x_3\} \). Consider the  metric $d \colon   X   \times X   \to [0,\infty)$ defined by
$$
d(x_1,x_2)=3, \quad d(x_1,x_3)=4, \quad d(x_2,x_3)=7
$$
 and the mapping \( T  \colon   X   \to  X   \) defined by
\[
T  x_1 = x_1, \quad T  x_2 = x_3, \quad T  x_3 = x_1.
\]
Consider also  $\delta \colon   X   \times  X   \to [0,\infty)$ defined  by

\begin{equation}\label{e44}
\delta(x_1,x_2)=3, \quad \delta(x_1,x_3)=1, \quad \delta(x_2,x_3)=5.
\end{equation}

It is obvious that  $\delta$ is not a metric, as it fails to satisfy triangle inequality, i.e., $\delta(x_2,x_3)=5>4=3+1=\delta(x_2,x_1)+\delta(x_1,x_3)$. 
    
At first, we show that $T $ does not satisfy the Banach contraction.  Since $$d(T  x_1,T  x_2)=d(x_1,x_3)=4>3=d(x_1,x_2),$$ there exist no $\alpha  \in [0,1)$ satisfying Banach contraction.

Define the ratio
\begin{equation}\label{e22}
f(x,y)={\frac{d(Tx,Ty)+\delta(Tx,Ty)}{2}}:{\big(d(x,y)\delta(x,y)\big)^\frac{1}{2}}.
\end{equation}
Now, 

Case-1: For $(x,y)=(x_1,x_2)$ in (\ref{e22}), we have
    \begin{equation*}
        \begin{split}f(x_1,x_2)=&{\frac{d(x_1,x_3)+\delta(x_1,x_3)}{2}}:{\big(d(x_1,x_2)\delta(x_1,x_2)\big)^\frac{1}{2}}\\
        =& {\frac{4+1}{2}}:{(3\times3)^\frac{1}{2}}
        = \frac{5}{6}.
        \end{split}
    \end{equation*}

Case-2: For $(x,y)=(x_2,x_3)$ in (\ref{e22}), we have
    \begin{equation*}
        \begin{split}
f(x_2,x_3)=&{\frac{d(x_1,x_3)+\delta(x_1,x_3)}{2}}:{\big(d(x_2,x_3)\delta(x_2,x_3)\big)^\frac{1}{2}}\\
        =& {\frac{4+1}{2}}:{(7\times5)^\frac{1}{2}}=
        \frac{5}{2\sqrt{35}}.
        \end{split}
    \end{equation*}
    
Case-3:  For $(x,y)=(x_1,x_3)$ in (\ref{e22}), we have
    \begin{equation*}
f(x_1,x_3)={\frac{d(x_1,x_1)+\delta(x_1,x_1)}{2}}:{\big(d(x_1,x_3)\delta(x_1,x_3)\big)^\frac{1}{2}}=0.
    \end{equation*}

This shows that (\ref{am-gm-contr}) is satisfied for all $\frac{5}{6}\le \alpha \le 1$. Consequently, \( x_1 \) is the unique fixed point of \( T   \), validating our result.
\end{example}

Note that the choice of a particular $\delta$ is very important in Example~\ref{ex1} because any semimetric $\delta$ will not work. The $\delta$ given by~\eqref{e44} is specifically chosen to work with the mapping and metric space of the Example \ref{ex1} so that $T $ satisfies A.M.-G.M. contraction. As for instance, in Example \ref{ex1}, if we choose a different $\delta$, say
$$
\delta(x_1,x_2)=3, \quad \delta(x_1,x_3)=2, \quad \delta(x_2,x_3)=5.
$$
$T$ fails to satisfy A.M.-G.M. contraction because 
 $$    
 \frac{d(T  x_1,T  x_2)+\delta(T  x_1,T  x_2)}{2}= \frac{d(x_1,x_3)+\delta(x_1,x_3)}{2}
 $$
 $$
=\frac{4+2}{2}
=3
=(3\times3)^\frac{1}{2}
=\big(d(x_1,x_2)\delta(x_1,x_2)\big)^\frac{1}{2}.
$$
 The following simple example shows that Theorem~\ref{am-gm-contr th} works well also for Banach contractions. 
\begin{example}
Let $ X  = [0,1]$ and define $d(x,y) = |x - y|$. We consider a function $T \colon   X  \to  X $ defined by $T (x) = \frac{x}{2}$. Additionally, we define $\delta(x,y) = |x - y| + |T (x) - T (y)|$. It is clear that $\delta$ is a semimetric on $X\times X$. $T$ satisfies A.M.-G.M. contraction because 
    \begin{equation*}
            \frac{d(Tx, Ty) +\delta(Tx, T y)}{2} =\frac{5|x-y|}{8}=\frac{5}{4\sqrt{6}}\big(d(x,y)\delta(x,y)\big)^\frac{1}{2} .
    \end{equation*}
Hence, $x=0$ is the unique fixed point of $T$.
\end{example}

\begin{definition}
Let \( ( X ,d) \) be a metric space and  let  $\delta\colon  X\times X \to \mathbb{R}^+$ be a semimetric on $X$. 
We shall say that \emph{the semimetric $\delta$ is $d$-regular} if for every $x^*\in X$ and every sequence $(x_n)$ in $X$ such that $d(x_n,x^*)\to 0$ we have $d(x_n,x^*)\delta(x_n,x^*) \to 0$.
\end{definition}

\begin{example}\label{ex26}
It is clear that given metric space $(X,d)$ a semimetric   $\delta\colon  X\times X \to \mathbb{R}^+$ is $d$-regular in the following cases:
\begin{itemize}
  \item [(i)] The semimetric space $(X,\delta)$ is bounded.
  \item [(ii)]  For all $x, y \in  X $ the inequality $\delta(x, y)\le K d(x, y)$ holds with some  $K>0$.
  \item [(iii)]  $\delta(x, y)=\psi(d(x, y))$, where $\psi\colon [0,\infty)\to [0,\infty)$ is a continuous function at $0$ such that $\psi(0)=0$.

\end{itemize}
\begin{remark}
  We need the condition $\psi(0)=0$ in order to guarantee that $\delta(x,x)=0$.
\end{remark}
\end{example}

\begin{lemma}\label{lem2.1}
Let $(X ,d)$ be a metric space and $(X,\delta)$ be a semimetric space. Then any A.M.-G.M. contraction $T\colon X \to X$ is continuous if the semimetric $\delta$ is $d$-regular.
\end{lemma}

\begin{proof}
Let $(x_n)$ be a sequence in $ X $ such that $d(x_n, x^*) \to 0$. 
Now, using (\ref{am-gm-contr}) we get
	\begin{equation*}
		\frac{d(T  x_n,T  x^*)}{2} \le\frac{d(T  x_n,T  x^*)+\delta(T  x_n,T  x^*)}{2} \le \alpha  \left(d(x_n,x^*)\delta(x_n,x^*)\right)^\frac{1}{2}.
	\end{equation*}
Hence, 
    \begin{equation*}
		d(T  x_n,T  x^*) \le 2\alpha  \left(d(x_n,x^*)\delta(x_n,x^*)\right)^\frac{1}{2}.
	\end{equation*}
Letting $n \rightarrow \infty$, by $d$-regularity of the semimetric $\delta$, we get $d(T  x_n,T  x^*) \rightarrow 0$, which completes the proof.
\end{proof}

The continuity condition of T in Theorem~\ref{am-gm-contr th} may be replaced by the $d$-regularity of the semimetric $\delta$.
\begin{corollary}
Let $(X, d)$ be a metric space, $(X, \delta)$ be a semimetric space, and let $T\colon X \to X$ be an A.M.-G.M. contraction. Then $T$ is a Picard operator provided that $\delta$ is $d$-regular.
\end{corollary}

Condition (iii) of Example~\ref{ex26} and Lemma~\ref{lem2.1} lead immediately to the following corollary of Theorem~\ref{am-gm-contr th}. 

\begin{corollary}\label{c210}
Let $( X ,d)$ be a complete metric space and $T  \colon   X  \rightarrow  X  $ be a mapping such that
\begin{equation}\label{e12}
\frac{d(T  x, T  y)+\psi(d(T  x,T  y))}{2} \leq \alpha  \big(d(x, y)\psi(d(x,y))\big)^\frac{1}{2},
\end{equation}
 for all   $x, y \in  X  $ and some $0\leq\alpha< 1$, where $\psi\colon [0,\infty)\to [0,\infty)$ is a continuous function at $0$ such that $\psi(0)=0$. Then $T$ is a Picard operator.
\end{corollary}

Inequality~(\ref{e12}) yields its most meaningful form when $\psi$ is chosen as a power function, $\psi(x)=kx^{\gamma}$, $k>0$, $\gamma>0$. In the partial cases $k=1$ and $\gamma=1,3,5,...$ inequality~(\ref{e12}) has the most elegant form.

In the case $k=\gamma=1$ we have the following.
\begin{corollary}(Banach's fixed point theorem) Theorem~\ref{ThBanach} holds.
\end{corollary}
\begin{proof}
Alternatively, it suffices to set $\delta=d$ in Theorem~\ref{am-gm-contr th}.
\end{proof}

In the case $\gamma =3$ inequality~(\ref{e12}) is 

\begin{equation}\label{e91}
\frac{d(T  x, T  y)+kd^3(T  x,T  y)}{2} \leq \alpha \sqrt{k} d^2(x, y).
\end{equation}

It is not hard to see that the mappings defined by inequality~(\ref{e91}) is Banach contraction with the coefficient $\sqrt{\alpha}$. Indeed, applying~(\ref{am-gm}) to the left hand side of~(\ref{e91}) we get
$$
\sqrt{kd(T  x, T  y)d^3(T  x,T  y)}\leq \alpha \sqrt{k} d^2(x, y)
$$
or
$$
d^2(T  x,T  y)\leq \alpha d^2(x, y)
$$
and
$$
d(T  x,T  y)\leq \sqrt\alpha d(x, y).
$$

In the case $\gamma =5$ we have 
\begin{equation}\label{e92}
\frac{d(T  x, T  y)+kd^5(T  x,T  y)}{2} \leq \alpha \sqrt{k} d^3(x, y), \,\, \text{etc}.
\end{equation}
Analogously, $T$ defined by~(\ref{e92}) is a  Banach contraction with the coefficient $\sqrt[3]{\alpha}$.

Inductively, the mapping $T$ becomes a Banach contraction with contraction coefficient $\sqrt[n]{\alpha}$ whenever $\gamma=2n-1$.
 
In the following example we construct a mapping $T$ that satisfies inequality~(\ref{e12}) with some $\psi$ but is not a Banach contraction.

\begin{example}
Let $(X,d)$ be a metric space such that $X = \{a, b, c\}$ and
\[
    d(a, b) = 1, \quad d(b, c) = 2, \quad d(a, c) = 3.
\]
Define $T : X \to X$  as
\[
    T(a) = b, \quad T(b) = c, \quad T(c) = c.
\]

First, we observe that $T$ does not satisfy the Banach Contraction Principle since for the pair $(a, b)$, the distance strictly increases:
$d(Ta, Tb) = d(b, c) = 2 > 1 = d(a, b)$.

Now, we show that $T$ nonetheless satisfies the conditions of Corollary 2.11. Let $\alpha = 0.5 \in [0, 1)$, and define the function $\psi : [0, \infty) \to [0, \infty)$ such that
\[
    \psi(t) = 4t(t - 2)^2.
\]

By definition, $\psi(0) = 0$ and $\psi$ is continuous at $t = 0$. We must verify that inequality~(\ref{e12}) for all $x, y \in X$. 
Since the inequality is trivially satisfied when $x = y$, we only need to check the distinct pairs:

 For $(x, y) = (a, b)$:
    \begin{multline*}
    1= \frac{2 + 0}{2} = \frac{d(b,c) + \psi(2)}{2} =
        \frac{d(Ta, Tb) + \psi(d(Ta, Tb))}{2}  \\
        \leqslant\alpha(d(a, b)\psi(d(a, b)))^{\frac{1}{2}} = 0.5(1 \cdot 4)^{\frac{1}{2}} = 1.
    \end{multline*}

 For $(x, y) = (a, c)$:
    \begin{multline*}
    1 = \frac{2 + 0}{2} = 
    \frac{d(b, c) + \psi(2)}{2}=
        \frac{d(Ta, Tc) + \psi(d(Ta, Tc))}{2}  \\
        \leqslant\alpha(d(a, c)\psi(d(a, c)))^{\frac{1}{2}} = 0.5(3 \cdot 12)^{\frac{1}{2}} = 0.5(36)^{\frac{1}{2}} = 3.
    \end{multline*}

 For $(x, y) = (b, c)$:
    \begin{multline*}
    0=
    \frac{0 + 0}{2} 
     = \frac{d(c, c) + \psi(0)}{2}= 
        \frac{d(Tb, Tc) + \psi(d(Tb, Tc))}{2} \\
        \leqslant\alpha(d(b, c)\psi(d(b, c)))^{\frac{1}{2}} = 0.5(2 \cdot 0)^{\frac{1}{2}} = 0.
    \end{multline*}
 \end{example}

\begin{problem}\label{op213}
Is it possible to construct a similar example of a mapping $T$ on the space $X$ containing accumulation points?
\end{problem}

\begin{problem}\label{op214}
  Under which additional conditions on the space $(X,d)$ or on the function $\psi$ the mapping $T$ in Corollary~\ref{c210} is reduced to the Banach contraction?
\end{problem}

\begin{problem}\label{op215}
Let $(X,d)$ be a complete metric space and let $T  \colon   X  \rightarrow  X $ be a Picard operator.
Does there always exist a semimetric $\delta$ which is $d$-regular such that the mapping $T$ is an A.M.-G.M. contraction? If not, under what additional conditions does such a semimetric $\delta$ exist? 
\end{problem}

\begin{problem}\label{op216}
Can Theorem~\ref{am-gm-contr th} be extended to means other than the arithmetic and geometric means?
\end{problem}

\section{Additive-type contractions}

In this section, we investigate the properties of additive type contractions and establish a fixed-point theorem for this class of mappings.

\begin{proposition}\label{prop_additive_banach}
Let $(X, d)$ be a metric space. Then, every additive-type contraction with the coefficient $\alpha \in (0,1)$ is a Banach contraction if for all $x, y \in X$ the inequality $\delta(x,y) \le \beta d(x,y)$ holds with $0 \le \beta < \frac{1-\alpha}{\alpha}$.
\end{proposition}

\begin{proof}
Let $T: X \to X$ be an additive-type contraction with the contraction coefficient $\alpha$. The following  inequality follows directly from~(\ref{additive-contr})
\begin{equation*}
    d(Tx,Ty) \le \alpha(d(x,y)+\delta(x,y)).
\end{equation*}
Since $\delta(x,y) \le \beta d(x,y)$, the inequality becomes
\begin{equation*}
    d(Tx,Ty) \le \alpha(d(x,y)+\beta d(x,y)) = \alpha(1+\beta)d(x,y),
\end{equation*}
which completes the proof, since $0 \le \alpha(1+\beta) < 1$.
\end{proof}

\begin{theorem}\label{thm_additive_picard}
Let $(X, d)$ be a complete metric space and $T: X \to X$ be an additive-type contraction. If $T$ is continuous in $(X,d)$, then $T$ is a Picard operator.
\end{theorem}

\begin{proof}
The proof is analogous to the proof of Theorem~\ref{am-gm-contr th}. 
\end{proof}

\begin{definition}\label{def_strongly_d_regular}
Let $(X, d)$ be a metric space and let $\delta\colon X \times X \to \mathbb{R}^+$ be a semimetric on $X$. We shall say that the semimetric $\delta$ is \textit{strongly d-regular} if for every $x^* \in X$ and every sequence $(x_n)$ in $X$ such that $d(x_n,x^*) \to 0$, we have $\delta(x_n,x^*) \to 0$.
\end{definition}

\begin{remark}\label{rem34}
It is clear that given a metric space $(X, d)$, a semimetric $\delta: X \times X \to \mathbb{R}^+$ is strongly $d$-regular if conditions (ii) and (iii) of Example~\ref{ex26} hold.
\end{remark}

\begin{lemma}\label{lem_additive_continuous}
Let $(X,d)$ be a metric space and $(X, \delta)$ be a semimetric space. Then any additive-type contraction $T\colon X \to X$ is continuous if the semimetric $\delta$ is strongly $d$-regular.
\end{lemma}

\begin{proof}
Let $(x_n)$ be a sequence in $X$ such that $d(x_n,x^*) \to 0$. From~(\ref{additive-contr}) we get
\begin{equation*}
    d(Tx_n,Tx^*) \le d(Tx_n,Tx^*)+\delta(Tx_n,Tx^*) \le \alpha(d(x_n,x^*)+\delta(x_n,x^*)).
\end{equation*}
Letting $n \to \infty$, by the strong $d$-regularity of the semimetric $\delta$, we get $\delta(x_n,x^*) \to 0$. Consequently,  $d(Tx_n,Tx^*) \to 0$, which completes the proof.
\end{proof}

The continuity condition of $T$ in Theorem~\ref{thm_additive_picard} may be replaced by the
strong $d$-regularity of the semimetric $\delta$.

\begin{corollary}
Let $(X, d)$ be a metric space, $(X, \delta)$ be a semimetric space, and let $T\colon X \to X$ be an A.M.-G.M. contraction. Then $T$ is a Picard operator provided that $\delta$ is strongly $d$-regular.
\end{corollary}

Remark~\ref{rem34}, condition (iii) of Example~\ref{ex26} and Lemma~\ref{lem_additive_continuous} lead immediately to the following.

\begin{corollary}\label{cor_additive_psi}
Let $(X, d)$ be a complete metric space and $T: X \to X$ be a mapping such that
\begin{equation}\label{eq_additive_psi}
    d(Tx,Ty)+\psi(d(Tx,Ty)) \le \alpha(d(x,y)+\psi(d(x,y)))
\end{equation}
for all $x, y \in X$ and some $0 \le \alpha < 1$, where $\psi\colon [0,\infty) \to [0,\infty)$ is a continuous function at $0$ such that $\psi(0)=0$. Then $T$ is a Picard operator.
\end{corollary}

\begin{corollary}\label{cor_additive_banach}
(Banach's fixed point theorem) Theorem~\ref{ThBanach} holds.
\end{corollary}

\begin{proof}
It suffices to set $\psi(x)=x$ in Corollary~\ref{cor_additive_psi} and simplify the inequality. Alternatively, it suffices to set $\delta=d$ in Theorem~\ref{thm_additive_picard}.
\end{proof}

\begin{example}
Let $(X,d)$ be a metric space such that  $X = \{a, b, c\}$ and
\begin{equation*}
    d(a, b) = 1, \quad d(b, c) = 2, \quad d(a, c) = 3,
\end{equation*}
Define the mapping $T : X \to X$ as
\begin{equation*}
    T(a) = b, \quad T(b) = c, \quad T(c) = c.
\end{equation*}

First, we observe that $T$ does not satisfy the Banach Contraction Principle. For the pair $(a, b)$, the distance strictly increases:
$ d(Ta, Tb) = d(b, c) = 2 > 1 = d(a, b)$.

Now, we show that $T$ nonetheless satisfies the conditions of Corollary~\ref{cor_additive_psi}. Let $\alpha = 0.8 \in [0, 1)$, and define the function $\psi : [0, \infty) \to [0, \infty)$ such that
\begin{equation*}
    \psi(t) = 
    \begin{cases} 
      100 & \text{if } t = 1, \\
      0 & \text{otherwise.}
   \end{cases}
\end{equation*}
By definition, $\psi(0) = 0$ and $\psi$ is continuous at $t = 0$. We must verify that the inequality~(\ref{eq_additive_psi})
 for all $x, y \in X$. Since the inequality is trivially satisfied when $x = y$, we only need to check the distinct pairs:

For $(x, y) = (a, b)$:
    \begin{multline*}
    2 = 2 + 0 = d(Ta, Tb) + \psi(d(Ta, Tb)) \\
        \leqslant \alpha \big( d(a, b) + \psi(d(a, b)) \big) = 0.8(1 + 100) = 80.8.
    \end{multline*}

 For $(x, y) = (a, c)$:
    \begin{multline*}
    2 = 2 + 0 = d(b, c) + \psi(2) =
        d(Ta, Tc) + \psi(d(Ta, Tc)) \\
        \leqslant \alpha \big( d(a, c) + \psi(d(a, c)) \big) = 0.8(3 + 0) = 2.4.
    \end{multline*}

 For $(x, y) = (b, c)$:
    \begin{multline*}
    0 = d(c, c) + \psi(0) = 
        d(Tb, Tc) + \psi(d(Tb, Tc)) \\
        \leqslant \alpha \big( d(b, c) + \psi(d(b, c)) \big) = 0.8(2 + 0) = 1.6.
    \end{multline*}
Therefore, $T$ is a Picard operator satisfying Corollary~\ref{cor_additive_psi}, despite failing the Banach Contraction Principle on the metric space $(X, d)$.
\end{example}

\begin{remark}
Corollary~\ref{cor_additive_psi} describes a sufficiently large class of contractive mappings. One can observe that by defining a new distance $\rho(x,y) = d(x,y) + \psi(d(x,y))$ we can conclude that $T$ is a Banach contraction on a new space $(X,\rho)$. It is so, but the point is that in general case $(X,\rho)$ is a semimetric space but not metric. Recall that a function $f\colon [0, \infty) \to [0, \infty)$ is  \emph{metric-preserving} if for every metric space $(X, d)$, the composition $d_f = f \circ d$ is also a metric on $X$. Since the function $f(t) = t + \psi(t)$ is not obligatory subadditive ($f(a+b) > f(a) + f(b)$), which is a necessary condition for preservation of the metric, see~\cite{Cor99}, there is no guarantee that $(X,\rho)$ is a metric space. Moreover, note that there exist no fixed point theorems establishing Banach contraction principle in an arbitrary semimetric space, such theorem can exist only under some additional conditions, see e.g.~\cite{PSB24}.
\end{remark}

\textbf{Connections with perturbed metric spaces.} In~\cite{JS25} M. Jleli and B. Samet introduced  the notion of a perturbed metric space.

\begin{definition}\label{d310}
Let $D, P : X \times X \to [0, \infty)$ be two given mappings. We say that $D$ is a perturbed metric on $X$ with respect to $P$, if
\begin{align*}
    d = D - P\colon X \times X &\to \mathbb{R}, \\
    (x, y) &\mapsto D(x, y) - P(x, y)
\end{align*}
is a metric on $X$. We call $P$ a perturbed mapping, $d = D - P$ an exact metric, and $(X, D, P)$ a perturbed metric space.
\end{definition}
The mapping $T$ is a \textit{perturbed continuous mapping}, if $T$ is continuous with respect to the exact metric $d$. We say that $(X,D,P)$ is a complete perturbed metric space, if $(X, d)$ is a complete metric space.
The following theorem was proved in~\cite[Theorem~3.1]{JS25}.
\begin{theorem}\label{t311}
Let $(X, D, P)$ be a complete perturbed metric space and $T : X \to X$ be a given mapping. Assume that the following conditions hold:
\begin{itemize}
    \item[(i)] $T$ is a perturbed continuous mapping;
    \item[(ii)] There exists $\lambda \in (0, 1)$ such that
    \begin{equation*}
        D(Tu, Tv) \leq \lambda D(u, v) 
    \end{equation*}
    for all $u, v \in X$.
\end{itemize}
Then, $T$ admits one and only one fixed point.
\end{theorem}

Note that Theorem~\ref{t311} is an equivalent of Theorem~\ref{thm_additive_picard} in the case $D$ and $P$ from Definition~\ref{d310} are restricted to be semimetrics. Indeed, if we set $D=d+\delta$ in Definition~\ref{d16}, then $(X,D,\delta)$ is a perturbed metric space, since  $d$ from Definition~\ref{d16} ($d=D-\delta$) is an exact metric.
According to conditions of Theorem~\ref{t311} $T$ is continuous on $(X,d)$ (which satisfies condition (i) of Theorem~\ref{t311}). According to Definition~\ref{d16} $T$ is a Banach  contraction on $(X,D)$ (which satisfies condition (ii) of Theorem~\ref{t311}).

\begin{problem}\label{op31}
Consider analogues of Open problems~\ref{op213},~\ref{op214} and~\ref{op215} for additive type contractions.
\end{problem}

\section{Multiplicative-type contractions}

In this section, we study the properties of multiplicative type contractions and establish a fixed-point theorem for these mappings.

\begin{proposition}\label{prop_multi_banach}
Let $(X, d)$ be a metric space. Then, every multiplicative-type contraction with the coefficient $\alpha \in [0,1)$ is a Banach contraction if there exist constants $c_2 \ge c_1 > 0$ such that for all $x, y \in X$, $x\neq y$, the inequality $c_1 \le \delta(x,y) \le c_2$ holds, and $\alpha \frac{c_2}{c_1} < 1$.
\end{proposition}

\begin{proof}
Let $T\colon X \to X$ be a multiplicative-type contraction with the contraction coefficient $\alpha$. By the definition of the contraction and the bounds on $\delta$, we have:
\begin{equation*}
    c_1 d(Tx,Ty) \le d(Tx,Ty)\delta(Tx,Ty) \le \alpha d(x,y)\delta(x,y) \le \alpha c_2 d(x,y).
\end{equation*}
Dividing both sides by $c_1$, we obtain:
\begin{equation*}
    d(Tx,Ty) \le \left(\alpha \frac{c_2}{c_1}\right) d(x,y),
\end{equation*}
which completes the proof, since $0 \le \alpha \frac{c_2}{c_1} < 1$.
\end{proof}

\begin{remark}
Unlike additive-type or A.M.-G.M. contractions where bounding $\delta$ strictly from above is sufficient, isolating the metric $d$ in a multiplicative formulation requires $\delta$ to be bounded away from zero to prevent the product from collapsing trivially. 
\end{remark}

\begin{definition}\label{def_d_lower_bounded}
Let $(X, d)$ be a metric space and let $\delta\colon X \times X \to \mathbb{R}^+$ be a semimetric on $X$.  We say that the semimetric $\delta$ is \textit{d-lower bounded} if there exist constants $c > 0$ and $\gamma > 0$ such that $\delta(x,y) \ge c \, d(x,y)^\gamma$ for all $x, y \in X$.
\end{definition}

\begin{lemma}\label{lem_multi_continuous}
Let $(X,d)$ be a metric space and $(X, \delta)$ be a semimetric space.  If the semimetric $\delta$ is both $d$-regular and $d$-lower bounded, then any multiplicative-type contraction mapping $T\colon X \to X$ is continuous.
\end{lemma}

\begin{proof}
Let $(x_n)$ be a sequence in $X$ such that $d(x_n,x^*) \to 0$. Using the multiplicative-type contraction condition~(\ref{multi-contr}) and the $d$-lower bound property, we have:
\begin{equation*}
    c \, d(Tx_n,Tx^*)^{1+\gamma} \le d(Tx_n,Tx^*)\delta(Tx_n,Tx^*) \le \alpha d(x_n,x^*)\delta(x_n,x^*).
\end{equation*}
Letting $n \to \infty$, by the $d$-regularity of the semimetric $\delta$, the right-hand side tends to $0$. Consequently, $c \, d(Tx_n,Tx^*)^{1+\gamma} \to 0$. Since $c > 0$ and $\gamma > 0$, this implies $d(Tx_n,Tx^*) \to 0$, completing the proof.
\end{proof}

\begin{theorem}\label{thm_multi_picard}
Let $(X, d)$ be a complete metric space and $T: X \to X$ be a multiplicative-type contraction. If $\delta$ is both $d$-regular and $d$-lower bounded, then $T$ is a Picard operator.
\end{theorem}

\begin{proof}
Suppose that $x_0 \in X$. We define a sequence $(x_n)$ for all $n>0$ by $x_n = T^n x_0$. 
Taking $x=x_0$ and $y=x_1$ in inequality~(\ref{multi-contr}), we have
\begin{equation*}
    d(x_1,x_2)\delta(x_1,x_2) \le \alpha d(x_0,x_1)\delta(x_0,x_1).
\end{equation*}
Taking $x=x_1$ and $y=x_2$ in the same inequality, we have
\begin{equation*}
    d(x_2,x_3)\delta(x_2,x_3) \le \alpha d(x_1,x_2)\delta(x_1,x_2) \le \alpha^2 d(x_0,x_1)\delta(x_0,x_1).
\end{equation*}
Continuing in this way, we deduce by induction that
\begin{equation*}
    d(x_n,x_{n+1})\delta(x_n,x_{n+1}) \le \alpha^n d(x_0,x_1)\delta(x_0,x_1).
\end{equation*}
Because $\delta$ is $d$-lower bounded, $\delta(x_n,x_{n+1}) \ge c \, d(x_n,x_{n+1})^\gamma$. Substituting this into the left-hand side yields
\begin{equation*}
    c \, d(x_n,x_{n+1})^{1+\gamma} \le \alpha^n d(x_0,x_1)\delta(x_0,x_1).
\end{equation*}
Solving for $d(x_n,x_{n+1})$, we obtain
\begin{equation}\label{eq_multi_distance_bound}
    d(x_n,x_{n+1}) \le \left( \frac{d(x_0,x_1)\delta(x_0,x_1)}{c} \right)^{\frac{1}{1+\gamma}} \left( \alpha^{\frac{1}{1+\gamma}} \right)^n.
\end{equation}
Let $M = \left( \frac{d(x_0,x_1)\delta(x_0,x_1)}{c} \right)^{\frac{1}{1+\gamma}}$ and $\lambda = \alpha^{\frac{1}{1+\gamma}}$. Since $\alpha \in [0, 1)$ and $\gamma > 0$, it follows strictly that $\lambda \in [0, 1)$. Thus, equation \eqref{eq_multi_distance_bound} simplifies to:
\begin{equation*}
    d(x_n,x_{n+1}) \le M \lambda^n.
\end{equation*}
Using the triangle inequality, for any $m > n$, we have:
\begin{align*}
d(x_n,x_m) &\le d(x_n,x_{n+1}) + d(x_{n+1},x_{n+2}) + \dots + d(x_{m-1},x_m)\\
&\le M \lambda^n + M \lambda^{n+1} + \dots + M \lambda^{m-1} \\
&\le M \lambda^n \frac{1 - \lambda^{m-n}}{1 - \lambda} \le \frac{M \lambda^n}{1 - \lambda}.
\end{align*}
Since $\lambda < 1$, as $n \to \infty$, $\lambda^n \to 0$. This proves that $(x_n)$ is a Cauchy sequence. Since $(X, d)$ is complete, there exists $x^* \in X$ such that $x_n \to x^*$.

By Lemma \ref{lem_multi_continuous}, the mapping $T$ is continuous, implying $Tx^* = x^*$. 

To prove uniqueness, suppose $u$ and $v$ are distinct fixed points ($Tu = u \ne v = Tv$). Then by~(\ref{multi-contr}) we have
\begin{equation*}
    d(u,v)\delta(u,v) \le \alpha d(u,v)\delta(u,v).
\end{equation*}
Since $u \ne v$, $d(u,v) > 0$ and $\delta(u,v) > 0$. Thus, $d(u,v)\delta(u,v) > 0$. Dividing both sides by this strictly positive value yields $1 \le \alpha$, which contradicts $\alpha \in [0, 1)$. Hence, $u=v$, and the fixed point is unique.
\end{proof}

\begin{remark}
Letting $m \to \infty$ in the sequence estimation, the distance from the fixed point $x^*$ to the $n$-th iterate of $T$ is bounded by:
\begin{equation*}
    d(x_n,x^*) \le \frac{M \lambda^n}{1-\lambda},
\end{equation*}
where $M = \left( c^{-1} d(x_0,x_1)\delta(x_0,x_1) \right)^{\frac{1}{1+\gamma}}$ and $\lambda = \alpha^{\frac{1}{1+\gamma}}$.
\end{remark}

\begin{corollary}\label{cor_multi_psi}
Let $(X, d)$ be a complete metric space and $T: X \to X$ be a mapping satisfying
\begin{equation}\label{eq_multi_psi}
    d(Tx,Ty)\psi(d(Tx,Ty)) \le \alpha d(x,y)\psi(d(x,y))
\end{equation}
for all $x, y \in X$, where $\alpha \in [0, 1)$ and $\psi\colon [0,\infty) \to [0,\infty)$ is a continuous function at $0$ such that $\psi(0)=0$ and $\psi(t) \ge c \, t^\gamma$ for some $c>0, \gamma > 0$. Then $T$ is a Picard operator.
\end{corollary}

\begin{example}
Let $\psi(x) = k x^\gamma$ with $k>0, \gamma>0$. Substituting this into \eqref{eq_multi_psi} yields
\begin{align*}
    d(Tx,Ty) \cdot k d(Tx,Ty)^\gamma &\le \alpha d(x,y) \cdot k d(x,y)^\gamma \\
    k d(Tx,Ty)^{1+\gamma} &\le \alpha k d(x,y)^{1+\gamma} \\
    d(Tx,Ty) &\le \alpha^{\frac{1}{1+\gamma}} d(x,y).
\end{align*}
Thus, multiplicative-type contraction in this case algebraically collapses the inequality into a standard Banach contraction with the coefficient $\lambda = \alpha^{\frac{1}{1+\gamma}}$. 
\end{example}

\begin{corollary}\label{cor_multi_exponential}
Let $(X, d)$ be a complete metric space and $T: X \to X$ be a mapping. If there exist constants $\alpha \in [0, 1)$ and $k > 0$ such that for all $x, y \in X$,
\begin{equation}
    d(Tx,Ty) \left( e^{k d(Tx,Ty)} - 1 \right) \le \alpha d(x,y) \left( e^{k d(x,y)} - 1 \right),
\end{equation}
then $T$ has a unique fixed point in $X$.
\end{corollary}

\begin{proof}
Consider the function $\psi(t) = e^{kt} - 1$. It is continuous at $0$ and $\psi(0) = e^0 - 1 = 0$. By the Taylor series expansion of the exponential function, $e^{kt} - 1 \ge kt$ for all $t \ge 0$. 
Thus, $\psi(t)$ satisfies the $d$-lower bound condition with $c = k$ and $\gamma = 1$. The result follows immediately from Corollary \ref{cor_multi_psi}.
\end{proof}

\begin{corollary}\label{cor_multi_logarithmic}
Let $(X, d)$ be a complete bounded metric space with diameter $M < \infty$, and let $T: X \to X$ be a self-mapping. If there exists $\alpha \in [0, 1)$ such that for all $x, y \in X$,
\begin{equation}
    d(Tx,Ty) \ln(1 + d(Tx,Ty)) \le \alpha d(x,y) \ln(1 + d(x,y)),
\end{equation}
then $T$ is a Picard operator.
\end{corollary}

\begin{proof}
Let $\psi(t) = \ln(1+t)$. This function is continuous at $0$, and $\psi(0) = 0$. Since $X$ is bounded by $M$, the maximum possible distance is $M$.  On the compact interval $[0, M]$, the function $f(t) = \frac{\ln(1+t)}{t}$ achieves its minimum value at $t=M$, which is strictly greater than zero. Let $c = \frac{\ln(1+M)}{M}$. Therefore, for all $t \in [0, M]$, we have $\ln(1+t) \ge c \cdot t$. This means $\psi(t)$ satisfies the $d$-lower bound condition with coefficient $c$ and $\gamma = 1$. Applying Corollary \ref{cor_multi_psi} yields the unique fixed point.
\end{proof}

\begin{problem}\label{op41}
Consider analogues of Open problems~\ref{op213},~\ref{op214} and~\ref{op215} for multiplicative type contractions.
\end{problem}

\section{Connections between the classes of mappings}
The following statement shows that, under the condition of comparability of the semimetric $\delta$ with the metric $d$ and a sufficiently small coefficient $\beta$ of the additive-type contraction, this mapping automatically becomes an  A.M.-G.M  contraction.

\begin{theorem} Let $(X,d)$ be a metric space and $\delta: X \times X \to [0,\infty)$ a semimetric. Suppose there exists a constant $K \ge 1$ such that for all $x,y \in X$ with $d(x,y) > 0$,
\begin{equation}\label{K-reg}
\frac{1}{K} \le \frac{\delta(x,y)}{d(x,y)} \le K.
\end{equation}
Assume that $T: X \to X$ is an additive-type contraction with coefficient $\beta \in [0,1)$:
\begin{equation}\label{ADC}
d(Tx,Ty) + \delta(Tx,Ty) \le \beta \bigl( d(x,y) + \delta(x,y) \bigr) \quad \forall x,y \in X.
\end{equation}
Then, if
$$
\beta < \frac{2\sqrt{K}}{1+K},
$$
the mapping $T$ is an A.M.-G.M contraction  with coefficient
$$
\alpha = \beta \cdot \frac{1+K}{2\sqrt{K}}.
$$
Moreover, $\alpha \in [0,1)$, and for all $x,y \in X$,
$$
\frac{d(Tx,Ty) + \delta(Tx,Ty)}{2} \le \alpha \sqrt{d(x,y)\,\delta(x,y)}.
$$
\end{theorem}

\begin{proof} The case $x=y$ is trivial. Let $x\neq y$.
Define the ratio
$$
R(x,y)=\frac{d(Tx,Ty) + \delta(Tx,Ty)}{2\sqrt{d(x,y)\,\delta(x,y)}}.
$$
From the additive-type contraction (\ref{ADC}),  we obtain
$$
R(x,y)=\frac{d(Tx,Ty) + \delta(Tx,Ty)}{2\sqrt{d(x,y)\,\delta(x,y)}}\leq \frac{\beta(d(x,y) + \delta(x,y))}{2\sqrt{d(x,y)\,\delta(x,y)}}.
$$
From (\ref{K-reg}) it follows that
$$\frac{d(x,y) + \delta(x,y)}{2\sqrt{d(x,y)\,\delta(x,y)}}\leq \frac{1+K}{2\sqrt{K}}\,.$$
Thus
$$
R(x,y)=\frac{d(Tx,Ty) + \delta(Tx,Ty)}{2\sqrt{d(x,y)\,\delta(x,y)}}\leq \frac{\beta(1+K)}{2\sqrt{K}} .
$$
Setting $ \alpha = \beta \cdot \frac{1+K}{2\sqrt{K}}$, we obtain
$$
\frac{d(Tx, Ty) + \delta(Tx, Ty)}{2} \leq \alpha \sqrt{d(x, y) \delta(x, y)}.
$$
Since $ \beta < \frac{2\sqrt{K}}{1+K}$, we have $ \alpha < 1 $. This completes the proof.
\end{proof}

\medskip

In the following statement, it is shown that, under the condition of comparability of the metric and the semimetric on the images of the mapping and a sufficiently small coefficient of multiplicative-type contraction, this mapping automatically becomes an A.M.-G.M. contraction.

\begin{theorem}
Let $(X,d)$ be a metric space, $\delta$ a semimetric on $X$, and $T: X \to X$ a multiplicative-type contraction:
\begin{equation}\label{MC}
d(Tx,Ty) \cdot \delta(Tx,Ty) \le \beta \; d(x,y) \cdot \delta(x,y) \qquad \forall x,y \in X,
\end{equation}
with $\beta \in [0,1)$. Suppose there exists a constant $K_* \ge 1$ such that for all $x,y \in X$ the following comparison condition on the images holds:
\begin{equation}\label{C*}
\frac{1}{K_*}\, \delta(Tx,Ty) \le d(Tx,Ty)\le K_*\delta(Tx,Ty)\, .
\end{equation}
If $\beta < \frac{4K_*}{(K_* + 1)^2}$,  then $T$ is an A.M.-G.M. contraction:
\begin{equation}
\frac{d(Tx,Ty) + \delta(Tx,Ty)}{2} \le \alpha \; \sqrt{d(x,y)\,\delta(x,y)} \qquad \forall x,y \in X,
\end{equation}
with
$$
\alpha = \frac{K_* + 1}{2\sqrt{K_*}} \cdot \sqrt{\beta}\in [0,1).
$$
\end{theorem}

\begin{proof} The case $Tx = Ty$ is trivial.  Assume that  $Tx \neq Ty$.  Define the ratio  $t = \frac{d(Tx,Ty)}{\delta(Tx,Ty)}.$
By the comparison condition on the images (\ref{C*}), we have  $t \in [1/K_*, K_*]$.

Now express the arithmetic mean and the geometric mean in terms of $\delta(Tx,Ty)$ and $t$:
$$
\frac{d(Tx,Ty) + \delta(Tx,Ty)}{2} = \delta(Tx,Ty)\frac{t+1}{2},
$$
$$
\sqrt{d(Tx,Ty)\,\delta(Tx,Ty)} = \delta(Tx,Ty)\sqrt{t}.
$$
Hence,
\begin{equation}\label{eq*1}
\frac{d(Tx,Ty) + \delta(Tx,Ty)}{2\sqrt{d(Tx,Ty)\,\delta(Tx,Ty)}} = \frac{t+1}{2\sqrt{t}}.
\end{equation}

Consider the auxiliary function $\varphi(t) = \dfrac{t+1}{2\sqrt{t}}$ for $t > 0$. Its derivative is $\varphi'(t) = \dfrac{t-1}{4t^{3/2}}$.
Thus $\phi$ is strictly decreasing on $(0,1]$ and strictly increasing on $[1,\infty)$. Since $K_{*} \geq 1$, the interval $[1/K_{*}, K_{*}]$ contains $t = 1$. Consequently, the maximum of $\phi$ on this interval is attained at one of the endpoints. Evaluating:
$$
\varphi(K_*) = \frac{K_*+1}{2\sqrt{K_*}}, \qquad
\varphi(1/K_*) = \frac{1/K_*+1}{2\sqrt{1/K_*}} = \frac{K_*+1}{2\sqrt{K_*}}.
$$
Therefore, for all $t \in [1/K_*, K_*]$,
\begin{equation}\label{eq*2}
\frac{t+1}{2\sqrt{t}} \le \frac{K_*+1}{2\sqrt{K_*}}.
\end{equation}
Combining (\ref{eq*1}) and (\ref{eq*2}) yields
$$
\frac{d(Tx,Ty) + \delta(Tx,Ty)}{2\sqrt{d(Tx,Ty)\,\delta(Tx,Ty)}} \le \frac{K_*+1}{2\sqrt{K_*}}.
$$
Multiplying both sides by  $2\sqrt{d(Tx,Ty)\,\delta(Tx,Ty)}$ gives
\begin{equation}\label{eq*3}
d(Tx,Ty) + \delta(Tx,Ty) \le \frac{K_*+1}{\sqrt{K_*}} \; \sqrt{d(Tx,Ty)\,\delta(Tx,Ty)}.
\end{equation}
Now, from the multiplicative-type contraction assumption (\ref{MC}),
$$
d(Tx,Ty)\,\delta(Tx,Ty) \le \beta \; d(x,y)\,\delta(x,y).
$$
Hence,
\begin{equation}\label{eq*84}
\sqrt{d(Tx,Ty)\,\delta(Tx,Ty)} \le \sqrt{\beta} \; \sqrt{d(x,y)\,\delta(x,y)}.
\end{equation}
Combining (\ref{eq*3}) and (\ref{eq*84}), we obtain:
$$
d(Tx,Ty) + \delta(Tx,Ty) \le \frac{K_*+1}{\sqrt{K_*}} \sqrt{\beta} \; \sqrt{d(x,y)\,\delta(x,y)}.
$$
Finally, divide by 2 to obtain the A.M.-G.M. contraction:
$$
\frac{d(Tx,Ty) + \delta(Tx,Ty)}{2} \le \frac{K_*+1}{2\sqrt{K_*}} \sqrt{\beta} \; \sqrt{d(x,y)\,\delta(x,y)}.
$$
Denote
$$
\alpha = \frac{K_*+1}{2\sqrt{K_*}} \cdot \sqrt{\beta}.
$$
The condition $\alpha < 1$ is equivalent to
$\beta < \frac{4K_*}{(K_*+1)^2}.$
This completes the proof.
\end{proof}

The following theorem shows that, under comparability conditions, an additive contraction implies a multiplicative one provided the coefficient
$\beta$  is sufficiently small.

\begin{theorem}[Additive $\rightarrow$ Multiplicative]\label{t0.3}
Let $(X,d)$ be a metric space and $\delta$ a semimetric on $X$.
Suppose there exists $K \geq 1$ such that for all $x \neq y$,
\begin{equation}\label{AMC1}
\frac{1}{K} \leq \frac{\delta(x,y)}{d(x,y)} \leq K.
\end{equation}

Let $T: X \to X$ satisfy the additive contraction
\begin{equation}\label{ADCth3}
d(Tx,Ty) + \delta(Tx,Ty) \leq \beta \big( d(x,y) + \delta(x,y) \big) \quad \forall x,y \in X,
\end{equation}
with $\beta \in [0,1)$.
Then, if
$$
\beta < \frac{2\sqrt{K}}{1+K},
$$
the mapping $T$ is a multiplicative contraction:
$$
d(Tx,Ty) \cdot \delta(Tx,Ty) \leq \beta' \cdot d(x,y) \cdot \delta(x,y) \quad \forall x,y \in X,
$$
with $\displaystyle \beta' = \frac{\beta^2 (K+1)^2}{4K} < 1$.
\end{theorem}

\begin{proof}
Assume that all conditions of the theorem are satisfied. For arbitrary $x,y\in X$, consider the case $x\neq y$; if $x=y$, then both sides of the multiplicative inequality are zero and the inequality holds trivially. From the comparison condition (\ref{AMC1}), for any $x\neq y$ we have $\delta(x,y)/d(x,y)\in[1/K, K]$.

Consider the function $\phi(t)=\frac{t+1}{2\sqrt{t}}$ for $t>0$. Its derivative is $\phi'(t)=\frac{t-1}{4t^{3/2}}$, so $\phi$ is strictly decreasing on $(0,1]$ and strictly increasing on $[1,\infty)$. Since $K\ge 1$, the interval $[1/K, K]$ contains $t=1$. Therefore, the maximum of $\phi$ on this interval is attained at one of the endpoints. Direct computation gives
$$
\phi(K)=\frac{K+1}{2\sqrt{K}},\qquad \phi(1/K)=\frac{1/K+1}{2\sqrt{1/K}}=\frac{K+1}{2\sqrt{K}}.
$$
Hence, for all $t\in[1/K, K]$ we have $\phi(t)\le \frac{K+1}{2\sqrt{K}}$. Setting $t=\delta(x,y)/d(x,y)$, we obtain
\begin{equation}\label{prth3eq1}
\frac{d(x,y)+\delta(x,y)}{2\sqrt{d(x,y)\delta(x,y)}} \le \frac{K+1}{2\sqrt{K}}.
\end{equation}
Now apply the additive contraction condition (\ref{ADCth3}). For any $x,y\in X$,
$$
\frac{d(Tx,Ty)+\delta(Tx,Ty)}{2\sqrt{d(x,y)\delta(x,y)}}
\le \frac{\beta\bigl(d(x,y)+\delta(x,y)\bigr)}{2\sqrt{d(x,y)\delta(x,y)}}.
$$
Using inequality (\ref{prth3eq1}), we obtain
\begin{equation}\label{prth3eq2}
\frac{d(Tx,Ty)+\delta(Tx,Ty)}{2\sqrt{d(x,y)\delta(x,y)}}
\le \beta\cdot\frac{K+1}{2\sqrt{K}}.
\end{equation}
Next, apply the classical inequality between the arithmetic and geometric means to the numbers $d(Tx,Ty)$ and $\delta(Tx,Ty)$:
\begin{equation}\label{prth3eq3}
\sqrt{d(Tx,Ty)\,\delta(Tx,Ty)} \le \frac{d(Tx,Ty)+\delta(Tx,Ty)}{2}.
\end{equation}
Combining (\ref{prth3eq2}) and (\ref{prth3eq3}), we get
$$
\sqrt{d(Tx,Ty)\,\delta(Tx,Ty)}
\le \beta\cdot\frac{K+1}{2\sqrt{K}}\cdot\sqrt{d(x,y)\delta(x,y)}.
$$
Squaring both sides (all quantities are nonnegative) yields
$$
d(Tx,Ty)\,\delta(Tx,Ty)
\le \beta^2\cdot\frac{(K+1)^2}{4K}\cdot d(x,y)\,\delta(x,y).
$$
Denote $\beta' = \frac{\beta^2(K+1)^2}{4K}$. Then the last inequality takes the form
$$
d(Tx,Ty)\,\delta(Tx,Ty) \le \beta'\,d(x,y)\,\delta(x,y) \qquad \forall x,y\in X.
$$
It remains to verify $\beta' < 1$. From $\beta < \frac{2\sqrt{K}}{1+K}$ we get $\beta^2 < \frac{4K}{(1+K)^2}$, hence $\frac{\beta^2(1+K)^2}{4K} < 1$, i.e. $\beta' < 1$. This completes the proof.
\end{proof}

\medskip

The following theorem shows that, under comparability conditions on the images, a multiplicative contraction implies an additive one provided the coefficient
$\beta$ is sufficiently small.

\begin{theorem}[Multiplicative $\rightarrow$ Additive]\label{t0.4}
Let $(X,d)$ be a metric space and $\delta$ a semimetric on $X$.
Suppose there exists $K_{*} \geq 1$ such that for all $x \neq y$,

\begin{equation}\label{eqth4C*}
\frac{1}{K_{*}} \leq \frac{d(Tx,Ty)}{\delta(Tx,Ty)} \leq K_{*}.
\end{equation}
Let $T: X \to X$ satisfy the multiplicative contraction
\begin{equation}\label{Mth4*}
d(Tx,Ty) \cdot \delta(Tx,Ty) \leq \beta \; d(x,y) \cdot \delta(x,y) \quad \forall x,y \in X,
\end{equation}
with $\beta \in [0,1)$.
Then, if
$$
\beta < \frac{4K_{*}}{(K_{*}+1)^2},
$$
the mapping $T$ is an additive contraction:
$$
d(Tx,Ty) + \delta(Tx,Ty) \leq \beta' \big( d(x,y) + \delta(x,y) \big) \quad \forall x,y \in X,
$$
with $\displaystyle \beta' = \frac{K_{*} + 1}{2\sqrt{K_{*}}} \cdot \sqrt{\beta} < 1$.
\end{theorem}

\begin{proof}
Assume that all conditions of the theorem are satisfied. For arbitrary   $x,y\in X$, consider the case $Tx\neq Ty$; if  $Tx=Ty$, then the left-hand side of the additive inequality vanishes and the inequality holds trivially. From the comparison condition on the images (\ref{eqth4C*}), for any $x,y$ with $Tx\neq Ty$ we have $d(Tx,Ty)/\delta(Tx,Ty)\in[1/K_*, K_*]$.

Consider the function $\phi(t)=\frac{t+1}{2\sqrt{t}}$ for $t>0$. Hence, for all $t\in[1/K_*, K_*]$ we have $\phi(t)\le \frac{K_*+1}{2\sqrt{K_*}}$. Setting $t=d(Tx,Ty)/\delta(Tx,Ty)$, we obtain

$$
\frac{d(Tx,Ty)+\delta(Tx,Ty)}{2\sqrt{d(Tx,Ty)\delta(Tx,Ty)}} \le \frac{K_*+1}{2\sqrt{K_*}}.
$$
Hence,
\begin{equation}\label{eq1profMth4*}
d(Tx,Ty)+\delta(Tx,Ty) \le \frac{K_*+1}{\sqrt{K_*}}\sqrt{d(Tx,Ty)\delta(Tx,Ty)}.
\end{equation}
Next, using the multiplicative contraction condition (\ref{Mth4*}), we have
\begin{equation}\label{eq2profMth4*}
\sqrt{d(Tx,Ty)\delta(Tx,Ty)} \le \sqrt{\beta}\,\sqrt{d(x,y)\delta(x,y)}.
\end{equation}
Combining (\ref{eq1profMth4*}) and (\ref{eq2profMth4*}), we obtain:
$$
d(Tx,Ty)+\delta(Tx,Ty) \le \frac{K_*+1}{\sqrt{K_*}}\sqrt{\beta}\,\sqrt{d(x,y)\delta(x,y)}. 
$$
Now applying the classical inequality between the arithmetic and geometric means to the numbers $d(x,y)$ and $\delta(x,y)$,  we obtain
$$
d(Tx,Ty)+\delta(Tx,Ty) \le  \frac{K_*+1}{2\sqrt{K_*}}\sqrt{\beta}\,\bigl(d(x,y)+\delta(x,y)\bigr).
$$
Denote $\beta' = \frac{K_*+1}{2\sqrt{K_*}}\sqrt{\beta}$. Then the last inequality takes the form
$$
d(Tx,Ty)+\delta(Tx,Ty) \le \beta'\bigl(d(x,y)+\delta(x,y)\bigr) \qquad \forall x,y\in X.
$$
It remains to verify $\beta' < 1$. From $\beta < \frac{4K_*}{(K_*+1)^2}$ we get $\sqrt{\beta} < \frac{2\sqrt{K_*}}{K_*+1}$, hence $\frac{K_*+1}{2\sqrt{K_*}}\sqrt{\beta} < 1$, i.e. $\beta' < 1$.
This completes the proof.
\end{proof}

\medskip

The following theorem establishes the equivalence of additive and multiplicative contractions for a sufficiently small coefficient $\beta$.

\begin{theorem}[Criterion of equivalence of additive and multiplicative contractions]
\label{thm:0.6}
Let $(X,d)$ be a metric space and $\delta$ a semimetric on $X$.
Suppose there exist constants $K, K_* \ge 1$ such that for all $x,y \in X$ with $x \ne y$
$$
\frac{1}{K} \le \frac{\delta(x,y)}{d(x,y)} \le K,
$$
and for all $x,y \in X$ with $Tx \ne Ty$
$$
\frac{1}{K_*} \le \frac{d(Tx,Ty)}{\delta(Tx,Ty)} \le K_*.
$$
Define
$$
\beta_0 = \min\left\{\frac{2\sqrt{K}}{1+K},\ \frac{4K_*}{(K_*+1)^2}\right\}.
$$
Then for any $\beta < \beta_0$ the following statements are equivalent:
\begin{enumerate}
\item  $T$ is an additive-type contraction:
$$
d(Tx,Ty) + \delta(Tx,Ty) \le \beta\bigl(d(x,y) + \delta(x,y)\bigr) \quad \forall x,y \in X;
$$
\item  $T$ is a multiplicative-type contraction:
$$
d(Tx,Ty) \cdot \delta(Tx,Ty) \le \beta \; d(x,y) \cdot \delta(x,y) \quad \forall x,y \in X.$$
\end{enumerate}

Moreover, if (1) holds, then \(T\) is a multiplicative contraction with coefficient \(\frac{\beta^2(K+1)^2}{4K} < 1\);
if (2) holds, then \(T\) is an additive contraction with coefficient \(\frac{K_*+1}{2\sqrt{K_*}}\sqrt{\beta} < 1\).
\end{theorem}

\medskip

\begin{proof}
  The equivalence follows directly from Theorems~\ref{t0.3} and~\ref{t0.4}, since \(\beta < \beta_0\) implies both \(\beta < \frac{2\sqrt{K}}{1+K}\) and \(\beta < \frac{4K_*}{(K_*+1)^2}\).
\end{proof}

\section*{Declarations: }
	\subsection*{Ethics approval and consent to participate}
	Not Applicable.
	\subsection*{Consent for publication}
	Not Applicable.
	\subsection*{Availability of data and materials}
	Not Applicable.
	\subsection*{Competing interests}
	The authors declare that they have no competing interests.
	\subsection*{Funding} E. Petrov was supported by the National Research Foundation of Ukraine, project No.~2025.07/0369 ``Qualitative methods of nonlinear analysis of heterogeneous structures''.

\end{document}